\documentclass[11pt]{article}
\usepackage{amsmath,amsthm,amsfonts,amssymb,amscd, amsxtra}
\theoremstyle{plain}
\newtheorem{theorem}{Theorem}
\newtheorem{lemma}[theorem]{Lemma}
\newtheorem{definition}[theorem]{Definition}
\newtheorem{corollary}[theorem]{Corollary}
\newtheorem{proposition}[theorem]{Proposition}

\newtheorem{remark}[theorem]{Remark}

\newtheorem{method}{Method}
\newcommand{\beq}{\begin{equation}}
\newcommand{\eeq}{\end{equation}}
\newcommand{\beqa}{\begin{eqnarray}}
\newcommand{\eeqa}{\end{eqnarray}}
\newcommand{\beqas}{\begin{eqnarray*}}
\newcommand{\eeqas}{\end{eqnarray*}}

\def\argmin{\operatorname{argmin}}

\def\min{\operatorname{min}}
\def\max{\operatorname{max}}

\DeclareMathOperator{\grad}{grad}

\DeclareMathOperator{\diag}{diag}
\begin{document}

\title{Unconstrained steepest descent method for multicriteria  optimization on Riemmanian manifolds}

\author{
G. C. Bento\thanks{IME, Universidade Federal de Goi\'as,
Goi\^ania, GO 74001-970, BR (Email: {\tt glaydston@mat.ufg.br})}
\and
O. P. Ferreira
\thanks{IME, Universidade Federal de Goi\'as,
Goi\^ania, GO 74001-970, BR (Email: {\tt orizon@mat.ufg.br}).
The author was supported in part by CNPq Grant 302618/2005-8,  PRONEX--Optimization(FAPERJ/CNPq) and FUNAPE/UFG.}\and
P. R. Oliveira \thanks{COPPE-Sistemas, Universidade Federal do Rio de Janeiro,
Rio de Janeiro, RJ 21945-970, BR (Email: {\tt poliveir@cos.ufrj.br}).
This author was supported in part by CNPq.}
}
\date{October 29, 2010}

\maketitle
\begin{abstract}
In this paper we present a steepest descent method with Armijo's rule for multicriteria optimization in the Riemannian context. The well definedness of the sequence generated by the method is guaranteed. Under mild assumptions on the multicriteria function, we prove that each accumulation point (if they exist) satisfies first-order necessary conditions for Pareto optimality. Moreover, assuming quasi-convexity of the multicriteria function and non-negative curvature of the Riemannian manifold, we prove full convergence of the sequence to a Pareto critical.
\end{abstract}
{\bf Key words:} Steepest descent, Pareto optimality, Vector optimization,   Quasi-Fej\'er convergence, Quasi-convexity,  Riemannian manifolds.
\section{Introduction}
Consider o following minimization problem
\begin{eqnarray}\label{po:conv777}
\begin{array}{clc}
   & \min  F(p) \\
   & \textnormal{s.t.}\,\,\, p\in X.\\
\end{array}
\end{eqnarray}
In case that $F:\mathbb{R}^n\to\mathbb{R}$ and $X=\mathbb{R}^n$, the steepest descent method with Armijo's rule generates a sequence $\{p^k\}$ as follows 
\[
p^{k+1}=p^k+t_kv^k, \qquad v^k=-F'(p^k),\qquad k=0,1,\ldots,
\]
where 
\[
t_k=\max\{2^{-j}: F(p^k-tv^k)\leq F(p^k)+\beta v^k, j=0,1\ldots\}, 
\]
$\beta\in (0,1)$. If $F$ is  continuously differentiable, classic results assure only that any accumulation point of $\{p^k\}$, case there exist, are critical of $F$. This fact was generalized for multicriteria optimization by Fliege and Svaiter \cite{Benar2000}, namely, whenever the objective function is a vectorial function $F:\mathbb{R}^n\to\mathbb{R}^m$ and the partial order in $\mathbb{R}^m$ is the usual, i.e., the componet-wise order. Full convergence is assured under the assumption that the solution set of the problem $\eqref{po:conv777}$ is not-empty and $F:\mathbb{R}^n\to\mathbb{R}$ is  a convex function, see Burachik et al. \cite{Bur1995} (or, more generally, a quasi-convex function, see Kiwiel and Murty \cite{Kiwiel1996}), which  has been  generalized for vector optimization by Gra\~na Drummond and Svaiter \cite{Benar2005} (see also, Gra\~na Drummond and Iusem~\cite{Iusem2004}).

Extension of concepts, techniques as well as methods from Euclidean spaces to  Riemannian manifolds  is natural and, in general, nontrivial. In the last few years, such extension setting with purpose practical and theoretical has been the subject of many new research. Recent works dealing with this issue include \cite{Absil2007, Teboulle2004, Azagra2005, BP09, FS02, XFLN2006, FO98, Zhu2007, Chong Li2009, lili09-1, PP08, PO2009, wangli09-2, wangli09, wanglihu09}. The generalization of optimization methods from Euclidean space to Riemannian manifold have some important advantages. For example, constrained optimization problems can be seen as unconstrained one from the Riemannian geometry viewpoint (the set constrained is a manifold) and, in this case, we have an alternative possibility besides the projection idea for solving the problem. Moreover, nonconvex problems in the classical context may become convex through the introduction of an appropriate Riemannian metric (see, for example \cite{XFLN2006}).

The steepest descent method for the problem~\eqref{po:conv777}, in the particular case that $X=M$ ($M$ a Riemannian manifold) and $F:M\to\mathbb{R}$ is continuously differentiable has been studied by Udriste~\cite{U94}, Smith~\cite{S94} and Rapcs\'ak~\cite{Rap97} and partial convergence results were obtained. For the convex case the full convergence, using Armijo's rule,  has been generalized by da Cruz Neto et al. \cite{XLO1998}, in the particular case that $M$ has non-negative curvature. Regarding  to the same restrictive assumption on the manifold $M$, Papa Quiroz et al. \cite{PP08} generalized the full convergence result using generalized Armijo's rule for the quasiconvex case.

In this paper, following the ideas of Fliege and Svaiter \cite{Benar2000}, we generalize its converge results for multicriteria optimization to the Riemannian context. Besides, following the ideas of Gra\~na Drummond and Svaiter \cite{Benar2005}, we generalize the full convergence result for multicriteria optimizationin in the case that the multicriteria function is quasi-convex and the Riemannian manifold has non-negative curvature.

The organization of our paper is as follows. In Section~\ref{sec1} we define the notations and
list some results of Riemannian geometry to be used throughout this paper. In Section~\ref{sec2} we present the multicriteria problem, the first order optimality condition for it and some basic definitions related. In the Section~\ref{sec3} we state the Riemannian steepest  descent methods for solving  multicriteria problems and establish the well definition of the sequence generated for it. In Section~\ref{sec4} we prove a partial convergence result without any additional assumption on $F$ besides the continuity differentiable and, assuming quasi-convexity of $F$ and not-negative curvature for $M$, a full convergence result is presented. Finally, in Section~\ref{sec5} we present some examples of complete Riemannian manifolds with explicit geodesic curves and the steepest descent iteration of the sequence generated by the proposed method.

\section{Preliminaries on Riemannian geometry}\label{sec1}
In this section, we introduce some fundamental  properties and notations of Riemannian manifold. These basics facts can be found in any introductory book of Riemannian geometry, for example in \cite{MP92} and \cite{S96}.

Let $M$ be a $n$-dimentional connected manifold. We denote by $T_pM$ the $n$-dimensional {\it tangent space} of $M$ at $p$, by $TM=\cup_{p\in M}T_pM$ {\itshape{tangent bundle}} of $M$ and by ${\cal X}(M)$ the space of smooth vector fields over $M$. When $M$ is endowed with a Riemannian metric $\langle \,,\, \rangle$, with corresponding norm denoted by $\| \; \|$, then $M$ is now a Riemannian manifold. Recall that the metric can be used to define the length of piecewise smooth curves $\gamma:[a,b]\rightarrow M$ joining $p$ to $q$, i.e., such that $\gamma(a)=p$ and $\gamma(b)=q$, by
\[
l(\gamma)=\int_a^b\|\gamma^{\prime}(t)\|dt,
\]
and, moreover, by minimizing this length functional over the set of all such curves, we obtain a Riemannian distance $d(p,q)$ which induces the original topology on $M$. The metric induces a map $f\mapsto\grad f\in{\cal X}(M)$ which associates to each scalar function smooth over $M$ its gradient via the rule $\langle\grad f,X\rangle=d f(X),\ X\in{\cal X}(M)$. Let $\nabla$ be the Levi-Civita connection associated to $(M,{\langle} \,,\, {\rangle})$. A vector field $V$ along $\gamma$ is said to be {\it parallel} if $\nabla_{\gamma^{\prime}} V=0$. If $\gamma^{\prime}$ itself is parallel we say that $\gamma$ is a {\it geodesic}. Because the geodesic equation $\nabla_{\ \gamma^{\prime}} \gamma^{\prime}=0$ is a second order nonlinear ordinary differential equation, then the geodesic $\gamma=\gamma _{v}(.,p)$ is determined by its position $p$ and velocity $v$ at $p$. It is easy to check that $\|\gamma ^{\prime}\|$ is constant. We say that $ \gamma $ is {\it normalized} if $\| \gamma ^{\prime}\|=1$. The restriction of a geodesic to a  closed bounded interval is called a {\it geodesic segment}. A geodesic segment joining $p$ to $q$ in $ M$ is said to be {\it minimal} if its length equals $d(p,q)$ and this geodesic is called a {\it minimizing geodesic}. If $\gamma$ is a curve joining points $p$ and $q$ in $ M$ then, for each $t\in [a,b]$, $\nabla$ induces an linear isometry, relative to ${ \langle}\, ,\, {\rangle}$, $P_{\gamma(a)\gamma(t)}:T_{\gamma(a)}M\to T_{\gamma(t)}M$, the so-called {\it parallel transport} along $\gamma$ from $\gamma(a)$ to $\gamma(t)$. The inverse map of $P_{\gamma(a)\gamma(t)}$ is denoted by $P_{\gamma(a)\gamma(t)}^{-1}:T_{\gamma(t)}  M \to T_{\gamma(a)}M$. In the particular case  of $\gamma$ is the unique curve joining points $p$ and $q$ in $M$ then parallel transport along $\gamma$ from $p$ to $q$ is denoted by $P_{pq}:T_{p}M\to T_{q}M$.

A Riemannian manifold is {\it complete} if geodesics are defined for any
values of $t$. Hopf-Rinow's theorem asserts that if this is the case then
any pair of points, say $p$ and $q$, in $M$ can be joined by a (not
necessarily unique) minimal geodesic segment. Moreover, $( M, d)$ is a
complete metric space and bounded and closed subsets are compact. 
Take $p\in M$, the {\it exponential map} $exp_{p}:T_{p}  M \to M $ is defined by $exp_{p}v\,=\, \gamma _{v}(1,p)$. 

We denote by $R$ {\it the curvature tensor \/} defined by 
$R(X,Y)=\nabla_{X}\nabla_{Y}Z- \nabla_{Y}\nabla_{X}Z-\nabla_{[X,Y]}Z$, with $X,Y,Z\in{\cal X}(M)$, where $[X,Y]=YX-XY$. Then the {\it sectional curvature \/} with respect 
to $X$ and $Y$ is given by $K(X,Y)=\langle R(X,Y)Y , X\rangle /(||X||^{2}||Y||^{2}-\langle X\,,\,Y\rangle ^{2})$, where $||X||=\langle X,X\rangle ^{2}$.

In the subsection \ref{S:4.2} of this paper, we will be mainly interested
in Riemannian manifolds for which $K(X,Y)\geq 0$ for any $X,Y\in{\cal X}(M)$.
Such manifolds are referred to as {\it manifolds with nonnegative curvature}. A fundamental geometric property of this class of manifolds is that the
distance between geodesics issuing from one point is, at least locally,
bounded from above by the distance between the corresponding rays in the
tangent space. A global formulation of this general principle is the {\it 
law of cosines} that we now pass to describe. A {\it geodesic hinge} in $ M$ is a pair of normalized geodesics segment $\
\gamma _{1} $ and $\gamma _{2}$ such that $\gamma_{1}(0)=\gamma_{2}(0)$ and at
least one of them, say $\gamma_{1}$, is minimal. From now on $l_{1}=l(\gamma _{1})$, $
l_{2}=l(\gamma _{2})$, $l_{3}=d(\gamma_{1}(l_{1}),\gamma _{2}(l_{2}))$ and $\alpha = <\!\!\!)( \gamma _{1}^{\prime}(0),\gamma _{2}^{\prime}(0))$.
\begin{theorem}\label{law-cosines}
(Law of cosines) In a complete Riemannian manifold with nonnegative
curvature, with the notation introduced above, we have 
\begin{equation}
l_{3}^{\,2}\leq l_{1}^{\,2}+l_{2}^{\,2}- 2 l_{1} l_{2} \cos \alpha . 
\end{equation}
\end{theorem}
\begin{proof}
See \cite{MP92} and \cite{S96}.
\end{proof}
{\it In this paper $M$ will denote a complete n-dimensional Riemannian manifold.}

\section{The multicriteria problem}\label{sec2}
In this section we present the multicriteria problem, the first order optimality condition for it and some basic definitions related.

Let $I:=\{1,\ldots,m\}$, ${\mathbb R}^{m}_{+}=\{x\in {\mathbb R}^m: x_{i}\geq 0, j\in I\}$ and ${\mathbb R}^{m}_{++}=\{x\in {\mathbb R}^m: x_{j}> 0, j\in I \}$.
For $x, \, y \in {\mathbb R}^{m}_{+}$,
$y\succeq x$ (or $x \preceq y$) means that $y-x \in {\mathbb R}^{m}_{+}$
and $y\succ x$ (or $x \prec y$) means that
$y-x \in {\mathbb R}^{m}_{++}$.

Given a vector continuously differentiable function $F:M\to {\mathbb R}^m$, we consider the problem of finding a {\it Pareto optimum point} of F, i.e., a point $p^*\in M$ such that there exists no other $p\in M$ with $F(p)\preceq F(p^*)$ and $F(p) \neq F(p^*)$.  We denote this unconstrained problem  in the Riemannian context as
\begin{equation} \label{eq:mp}
\min_{p\in M} F(p).
\end{equation}
Let $F$ be given by $F(p):=\left(f_1(p), \ldots, f_m(p)\right)$. We denote the Riemannian jacobian of $F$ by  
\[
\grad F(p):=\left(\grad f_1(p), \ldots, \grad f_m(p) \right), \qquad p\in M,
\]
and the image of the Riemannian jacobian of $F$ at a point $p\in M$  by
\[
\mbox{Im} (\grad F(p)):=\left\{ \grad F(p)v=(\langle \grad f_1(p),v\rangle, \ldots, \langle \grad  f_m(p),v\rangle) : v\in T_pM \right\},  \qquad p\in M.
\]
Using above equality  the first order optimality condition for the problem \eqref{eq:mp} is stated as 
\begin{equation} \label{eq:oc} 
p\in M, \qquad \mbox{Im}(\grad F(p))\cap (-{\mathbb R}^{m}_{++})=\emptyset.
\end{equation}
\begin{remark}
Note that the  condition in \eqref{eq:oc}  generalizes to vector optimization the classical
condition $\grad F(p)=0$  for the scalar case, i.e.,  $m=1$. 
\end{remark}
In general, \eqref{eq:oc} is necessary, but no sufficient for optimality. So,  a point $p\in M$  satisfying  \eqref{eq:oc} is  called  { \it Pareto critical}. 
\section{Steepest descent methods for multicriteria problems}\label{sec3}
In this section we state the Riemannian steepest  descent methods for solving  multicriteria problems and establish the well definition of the sequence generated for it.

Let $p\in M$ be a point which  is not Pareto critical. Then there exists a direction $v\in T_pM$ satisfying
\[
\grad F(p)v \in -{\mathbb R}^{m}_{++},
\]
that is,  $\grad F(p)v \prec 0$. In this case, $v$ is called a {\it descent direction} for $F$ at $p$. 

For each $p\in M$,  we consider the following unconstrained optimization problem in the tangent plane $T_pM$
\begin{equation} \label{eq:dsdd}
\mathop{\min}_{v\in T_{p}M} \; \left\{\max_{i\in I}\langle\grad f_i(p),v\rangle+ (1/2)\|v\|^{2}\right\}, \quad \qquad I=\{1,\ldots,m\}.
\end{equation}
\begin{lemma}\label{dir:md1}
The unconstrained optimization problem in \eqref{eq:dsdd} has only one solution. Moreover, the vector $v$ is the solution of the problem in \eqref{eq:dsdd} if  only if there exist  $\alpha_i\geq 0$, $i\in I(p,v)$, such that 
\[
v=-\sum\limits_{i\in I(p,v)}\alpha_i\grad f_i(p), \qquad \sum\limits_{i\in I(p,v)}\alpha_i=1,
\]
where $I(p,v):=\{i\in I: \langle\grad f_i(p),v \rangle=\max_{i\in I}\langle\grad f_i(p),v \rangle\}$.
\end{lemma}
\begin{proof}
Since the function 
\[
T_{p}M\ni v\mapsto\max_{i\in I} \langle\grad f_i(p),v\rangle,
\]
is the maximum of linear functions in the linear space $T_{p}M$, it is convex. So, it is easy to see that the function 
\begin{equation}\label{eq:ofup}
T_{p}M\ni v\longmapsto\max_{i\in I}\langle\grad f_i(p),v\rangle + (1/2)\|v\|^{2},
\end{equation}
is strong convex, which implies that the problem in \eqref{eq:dsdd} has only one solution in $T_pM$ and the first statement is proved. 

From convexity of the function in \eqref{eq:ofup}, it is well know that $v$ is solution of the problem in \eqref{eq:dsdd} if only if 
\[
0\in\partial \left(\max_{i\in I}\langle\grad f_i(p),\,.\,\rangle + (1/2)\|\,.\,\|^{2} \right)(v),
\]
or  equivalently,
\[
-v\in \partial \left(\max_{i\in I}\langle\grad f_i(p),\,.\,\rangle\right)(v).
\]
Therefore, the second statement follows of the formula for the subdifferential of the maximum of convex functions (see \cite{HL93}, Volume I, Corollary VI.4.3.2).
\end{proof}
\begin{lemma}\label{dir:md2}
If $p\in M$ is not Pareto critical of $F$ and  $v$ is the solution of the problem in \eqref{eq:dsdd}, then 
\[
\max_{i\in I}\langle\grad f_i(p),v\rangle + (1/2)\|v\|^{2}<0.
\]
In particular, $v$ is a descent direction.
\end{lemma}
\begin{proof}
Since $p$ is not  Pareto critical, there exists $0\neq \hat v\in T_{p}M$ such that $ \grad F(p) \hat v\prec 0$. In particular,
\[
\beta=\max_{i\in I}\langle\grad f_i(p),\hat v\rangle<0.
\] 
As $-\beta/\|\hat v\|^2>0$, letting  $\bar{v}=(-\beta/\|\hat v\|^2)\hat v$ we obtain
\[
\max_{i\in I}\langle\grad f_i(p),\bar{v}\rangle + (1/2)\|\bar{v}\|^{2}=-\frac{\beta^2}{2\|\hat v\|^2}<0,
\]
Using that $v$  is the solution of the problem in \eqref{eq:dsdd},  the first part of the lemma  follows from last inequality. The second part of the lemma  is an immediate consequence of the first one.
\end{proof}

In view of the two previous lemmas and \eqref{eq:dsdd}  we define the steepest descent direction function for $F$ as follows.
\begin{definition}\label{def:sddf1}
The steepest descent direction  function for $F$ is defined as
\[
M\ni p\longmapsto v(p):= \argmin_{v\in T_pM} \left\{ \max_{i\in I}\left\langle\grad f_i(p),v\right\rangle + (1/2)\|v\|^{2}\right\}\in T_pM.
\]
\end{definition}

\begin{remark}
As an immediate consequence of Lemma \ref{dir:md1} it follows that the steepest descent direction for vector functions becomes the steepest descent direction when $m=1$. See, for example, \cite{XLO1998}, \cite{Munier2007}, \cite{Rap97}, \cite{S94} and \cite{U94}. In the case $M=\mathbb{R}^n$ we retrieve the steepest descent direction proposed in \cite{Benar2000}.
\end{remark}

The {\it steepest descent method with Armijo rule }  for solving the unconstrained  optimization problem \eqref{eq:mp} is as follows:

\begin{method}[Steepest descent method with Armijo rule]\label{algor1}
\hfill \vspace{.5cm}

\noindent
{\sc Initialization.} Take $\beta\in (0,\,1)$ and $p_0\in M$. Set $k=0$.\\
{\sc Stop criterion.} If $p^{k}$ is  Pareto critical STOP. Otherwise.\\
{\sc Iterative Step.} Compute the steepest descent direction $v^{k}$ for $F$ at $p^{k}$, i.e., 
\begin{equation} \label{eq:vk}
v^{k}:=v(p^{k}),
\end{equation}
and the steplength $t_k \in (0,1]$ is of the following way:
\begin{equation} \label{eq:sl}
 t_k := \max \left\{2^{-j} : j\in {\mathbb N},\, F \left(exp _{p^{k}}(2^{-j}v^k)\right)\preceq  F(p^{k})+ \beta 2^{-j}\,\grad F(p^{k})v^k\right\},
\end{equation}
and set
\begin{equation} \label{eq:xk}
p^{k+1}:=exp _{p^{k}}(t_{k}v^k),
\end{equation}
and GOTO {\sc Stop criterion}.
\end{method}
\begin{remark}
The steepest descent method for vector optimization in Riemannian manifolds  becomes the classical steepest descent method when $m=1$, which has appeared, for example, in \cite{XLO1998}, \cite{Rap97}, \cite{S94} and \cite{U94}.
\end{remark}

\begin{proposition} \label{p:wd}
The sequence $\{p^{k}\}$ generated by steepest descent method with Armijo rule is well defined.
\end{proposition}
\begin{proof}
Assume that $p^{k}$ is not  Pareto critical. From Definition \ref{def:sddf1} and Lemma \ref{dir:md1}, $v^k=v(p^{k})$ is well defined. Thus, for  proving the well definition of the method proposed it is enough proving well definition of the steplength. For this, first note that from Definition \ref{def:sddf1} and Lemma \ref{dir:md2} 
\[
\grad F(p^{k})v^k\prec 0.
\]
Since $F : M \to {\mathbb R}^m$ is a continuously differentiable vector function, $\grad F(p^{k})v^k\prec 0$ and $\beta\in (0, 1)$ we have
\[
\lim_{t\to 0^+}\frac{F \left(exp _{p^{k}}(tv^k)\right)-F(p^{k})}{t}=\grad F(p^{k})v^k\prec\beta \grad F(p^{k})v^k\prec 0.
\]
Therefore, it is straightforward to show that there exists $\delta\in (0, 1]$ such that 
\[
 F \left(exp _{p^{k}}(tv^k)\right)\prec  F(p^{k})+ \beta t\,\grad F(p^{k})v^k, \qquad t\in (0,\, \delta).
\]
As $\lim_{j\to \infty}2^{-j}=0$, last vector inequality implies that the  steplength \eqref{eq:sl} is well defined. Hence  $p^{k+1}$ is also well defined and the proposition is concluded.
\end{proof}

\section{Convergence analysis}\label{sec4}
In this section, following the ideas of \cite{Benar2000} we prove a partial convergence result without any additional assumption on $F$ besides the continuity differentiable. In the sequel, following \cite{Benar2005}, assuming quasi-convexity of $F$ and not-negative curvature for $M$, we extend to optimization of vector functions the full convergence result presented in \cite{XLO1998} and \cite{PP08}. It is immediate to see that, if the Method~\ref{algor1} terminates after a finite number of iterations, it terminates at a Pareto critical point. From now on,  we will assume that $\{p^{k}\}$, $\{v^k\}$ and $\{t_k\}$ are infinite sequences generated by Method~\ref{algor1}.

\subsection{Partial convergence result}
In this subsection we prove that every accumulation point of $\{p^{k}\}$ is a Pareto critical point. Before this, we prove the following preliminary fact that will be useful.

\begin{lemma}\label{lem:convv1}
The steepest descent direction function for $F$, $M\ni p\mapsto v(p) \in T_pM$, is continuous.
\end{lemma}
\begin{proof}
Let $\{q^k\}\subset M$ be a sequence which converges to $\bar{q}$ as $k$ goes to $+\infty$, and $U_{\bar{q}}\subset M$ a neighborhood of $\bar{q}$ such that $TU_{\bar{q}}\approx U_{\bar{q}}\times\mathbb{R}^n$. Since $\{q^k\}$ converges to $\bar{q}$ and  $TU_{\bar{q}}\subset TM$ is an open set, we assume that the whole sequence $\{(q^k,v(q^k))\}$ is in $TU_{\bar{q}}$. Define $v^{k}:=v(q^k)$. Combining Definition \ref{def:sddf1} with Lemma \ref{dir:md2} it is easy to see that 
\[
\|v^k\|\leq 2\max_{i\in I}\|\grad f_j(q^k)\|.
\]
As $F$ is continuously differentiable and $\{q^k\}$ is convergent, above inequality implies that the sequence $\{v^k\}$ is bounded. Let $\bar{v}$ be an accumulation point of the sequence $\{v^k\}$. From Definition~\ref{def:sddf1} and Lemma~\ref{dir:md1} we conclude that there exist $\alpha^{k}_i\geq 0$, $i\in I(q^k, v^{k})$, such that 
\begin{equation}\label{eq:contv1}
v^{k}=-\sum\limits_{i\in I(q^k, v^{k})}\alpha^k_i\grad f_i(q^k), \qquad \sum\limits_{i\in I(q^k, v^{k})}\alpha^{k}_i=1,\qquad k=0,1,\ldots.
\end{equation}
where $I(q^k, v^{k}):=\{i\in I: \langle\grad f_i(q^k),v^{k} \rangle=\max_{i\in I}\langle\grad f_i(q^k),v^{k} \rangle\}$. Using above constants  and the associated  indexes, define the sequence $\{\alpha^{k}\}$ as 
\[
\alpha^{k}:=(\alpha^{k}_1,\ldots,\alpha^{k}_m), \qquad \alpha^{k}_i=0, \quad i\in I\setminus I(q^k, v^{k}),\qquad k=0,1,\ldots.
\]
Let $\|\,.\,\|_{1}$ be the sum norm in $\mathbb{R}^{m}$. Since $\sum_{i\in I(q^k,v^k)}\alpha^{k}_i=1$, we have $\|\alpha^{k}\|_{1}=1$ for all $k$, which implies that the sequence $\{\alpha^{k}\}$ is bounded. Let $\bar{\alpha}$ be an accumulation point of the sequence $\{\alpha^{k}\}$. Let $\{v^{k_s}\}$ and $\{\alpha^{k_s}\}$  be  subsequences  of $\{v^{k}\}$ and $\{\alpha^{k}\}$ respectively, such that 
\[
\lim_{s  \to +\infty}v^{k_s}= \bar v, \qquad \lim_{s  \to +\infty}\alpha^{k_s}= \bar \alpha.  
\]
As the index set  $I$ is finite and $I(q^{k_s}, v^{k_s})\subset I$  for all $s$, we assume without loss of generality that
\begin{equation} \label{eq:eii}
I(q^{k_1}, v^{k_1})=I(q^{k_2}, v^{k_2})=...=\bar{I}.
\end{equation}
Hence, we conclude from \eqref{eq:contv1} and last equalities that
\[
v^{k_s}=-\sum\limits_{i\in \bar{I}}\alpha^{k_s}_i\grad f_i(q^{k_s}), \qquad \sum\limits_{i\in \bar{I}}\alpha^{k_s}_i=1,    \qquad s= 0, 1,  \ldots.
\]
Letting $s$ goes to $+\infty$ in the above equalities, we obtain
\begin{equation} \label{eq:ocsp1}
\bar{v}=\sum_{i\in\bar{I}}\bar{\alpha}_i\grad f_i(\bar{q}), \qquad \sum_{i\in\bar{I}}\bar{\alpha}_i=1. 
\end{equation}
On the other hand, $I(q^{k_s}, v^{k_s})=\{i\in I: \langle\grad f_i(q^{k_s}),v^{k_s} \rangle=\max_{i\in I}\langle\grad f_i(q^{k_s}), v^{k_s} \rangle\}$. So, equation  \eqref{eq:eii} implies that
\[
\langle\grad f_i(q^{k_s}),v^{k_s} \rangle=\max_{i\in I}\langle\grad f_i(q^{k_s}), v^{k_s}\rangle, \qquad   i\in  \bar{I},  \qquad s= 0, 1,  \ldots.
\]
Using continuity of $\grad F$ and last equality we have
\[
\langle\grad f_i(\bar q), \bar v\rangle=\max_{i\in I}\langle\grad f_i(\bar q), \bar v\rangle, \qquad   i\in  \bar{I}.
\]
From the definition of $I(\bar{q}, \bar{v})$ we obtain $\bar{I}\subset I(\bar{q}, \bar{v})$. Therefore, combining again Definition \ref{def:sddf1} with Lemma \ref{dir:md1} and \eqref{eq:ocsp1}, we conclude that $\bar{v}=v(\bar{q})$ and  the desired result is proved.
\end{proof}
In the next result, we just use that $F$ is continuously  differentiable to  assure that the sequence of the functional values of the sequence $\{p^{k}\}$, $\{F(p^{k})\}$, it is monotonous decreasing and that their accumulation points are critical Pareto.

\begin{theorem}\label{resconv1} The following statements there hold:
\begin{itemize}
	\item [i)] $\{F(p^{k})\}$ is decreasing;
	\item [ii)] Each accumulation point of the sequence $\{p^{k}\}$ is a Pareto critical point.
\end{itemize}
\end{theorem}
\begin{proof}
The iterative step in the Method~\ref{algor1} implies that 
\begin{equation}\label{eq:dir:md3}
F(p^{k+1})\preceq F(p^{k})+\beta t_k\grad F(p^{k})v^k, \qquad p^{k+1}=\exp_{p^{k}}t_kv^k, \qquad k=0, 1, \ldots .
\end{equation}
Since $\{p^{k}\}$ is a infinite sequence, for all $k$, $p^{k}$ is not Pareto critical of $F$. Thus, item {\it i} follows from definition of $v^k$ together with Definition \ref{def:sddf1}, Lemma~\ref{dir:md2} and last vector inequality.

Let $\bar{p}\in M$ be an accumulation point of the sequence $\{p^{k}\}$ and $\{p^{k_s}\}$ a subsequence of $\{p^{k}\}$ such that $\lim_{s \to +\infty}p^{k_s}=\bar{p}$. Since $F$ is continuous and $\lim_{s \to +\infty}p^{k_s}=\bar{p}$ we have $\lim_{s \to +\infty}F(p^{k_s})=F(\bar{p})$. So, taking into account that $\{F(p^{k})\}$ is a decreasing sequence and has $F(\bar{p})$ as an accumulation point, it is easy to  conclude that the whole sequence $\{F(p^{k})\}$ converges to $F(\bar{p})$.
Using the equation \eqref{eq:dir:md3}, Definition~\ref{def:sddf1} and Lemma~\ref{dir:md2}, we conclude that
\[
F(p^{k+1})- F(p^{k})\preceq \beta t_k\grad F(p^{k})v^k\preceq 0, \qquad k=0, 1, \ldots.
\]
Since $\lim_{s \to +\infty}F(p^{k})=F(\bar{p})$, last inequality implies that
\begin{equation}\label{eq:convfv2}
\lim_{k\to+\infty}\beta t_k\grad F(p^{k})v^k=0.
\end{equation}
As $\{p^{k_s}\}$ converges to $\bar{p}$, we  assume that $\{(p^{k_s},v^{k_s})\}\subset TU_{\bar{p}}$, where $U_{\bar{p}}$ is a neighborhood of $\bar{p}$ such that $TU_{\bar{p}}\approx U_{\bar{p}}\times\mathbb{R}^n$. Moreover, as the sequence $\{t_k\}\subset (0,1]$ has an accumulation point $\bar{t}\in [0,1]$, we assume without loss of generality that $\{t_{k_s}\}$ converges to $\bar{t}$. We have two possibilities to consider:
\begin{itemize}
	\item [{\bf a)}] $\bar{t}>0$;
	\item [{\bf b)}] $\bar{t}=0$.
\end{itemize}
Assume that the item ${\bf a}$ holds. In this case, from \eqref{eq:convfv2}, continuity of $\grad F$, \eqref{eq:vk} and Lemma \ref{lem:convv1}, we obtain
\[
\grad F(\bar{p})v(\bar{p})=0,
\]
which implies that 
\begin{equation}\label{eq:dir:md4}
\max_{i\in I}\langle\grad f_i(\bar{p}),v(\bar{p})\rangle=0.
\end{equation}
On the other hand, from Definition \ref{def:sddf1} together with Lemma \ref{dir:md2}, 
\[
\max_{i\in I}\langle\grad f_i(p^{k_s}),v^{k_s}\rangle + (1/2)\|v^{k_s}\|^{2}<0.
\]
Letting $s$ goes to $+\infty$ in the above inequalities and using Lemma \ref{lem:convv1} combined with the continuity of  $\grad F$ and equality \eqref{eq:dir:md4}, we conclude that
\[
\max_{i\in I}\langle\grad f_i(\bar{p}),v(\bar{p})\rangle + (1/2)\|v(\bar{p})\|^{2}=0.
\]  
Hence, it follows from last equality, Definition \ref{def:sddf1} and Lemma \ref{dir:md2} that $\bar{p}$ is a Pareto critical. 

Now, assume that the item ${\bf b}$ holds. Since for all $s$ $p^{k_s}$ is not a Pareto critical, we have
\[
\max_{i\in I}\langle\grad f_i(p^{k_s}),v^{k_s}\rangle\leq \max_{i\in I}\langle\grad f_i(p^{k_s}),v^{k_s}\rangle+(1/2)\|v^{k_s}\|^2<0,
\]
where the last inequality is consequence from Definition \ref{def:sddf1} together with Lemma \ref{dir:md2}. Hence, letting $s$ goes to $+\infty$ in the last inequalities,  using \eqref{eq:vk} and Lemma~\ref{lem:convv1} we obtain
\begin{equation}\label{eq:conv1}
\max_{i\in I}\langle\grad f_i(\bar{p}),v(\bar{p})\rangle\leq \max_{i\in I}\langle\grad f_i(\bar{p}),v(\bar{p})\rangle+(1/2)\|v(\bar{p})\|^2\leq0.
\end{equation}
Take $r\in\mathbb{N}$. Since $\{t_{k_s}\}$ converges to $\bar{t}=0$, we conclude that for $s$ large enough,
\[
t_{k_s}<2^{-r}.
\]
From \eqref{eq:sl} this means that the Armijo condition \eqref{eq:dir:md3} is not satisfied for $t=2^{-r}$, i.e.,
\[
F(\exp_{p^{k}}(2^{-j}v^{k_s}))\npreceq F(p^{k_s})+\beta 2^{-r}\grad F(p^{k_s})v^{k_s}, 
\]
which  means that there exists at least one $i_0\in I$ such that 
\[
f_{i_0}(\exp_{p^{k_s}}(2^{-r}v^{k_s}))> f_{i_0}(p^{k_s})+\beta 2^{-r}\langle\grad f_{i_0}(p^{k_s}),v^{k_s}\rangle.
\]
Letting $s$ goes to $+\infty$ in the above inequality,  taking into account that $\grad F$ and $\exp$ are  continuous  and using Lemma~ \ref{lem:convv1}, we obtain
\[
f_{i_0}(\exp_{\bar{p}}(2^{-r}v(\bar{p})))\geq f_{i_0}(\bar{p})+\beta 2^{-r}\langle\grad f_{i_0}(\bar{p}),v(\bar{p})\rangle.
\]
Last inequality is equivalent to
\[
\frac{f_{i_0}(\exp_{\bar{p}}(2^{-r}v(\bar{p})))-f_{i_0}(\bar{p})}{2^{-r}}\geq \beta\langle\grad f_{i_0}(\bar{p}),v(\bar{p})\rangle,
\]
which, letting $r$ goes to $+\infty$ and using that $0<\beta<1$ yields $\langle\grad f_{i_0}(\bar{p}),v(\bar{p})\rangle \geq 0$. Hence, 
\[
\max_{i\in I}\langle\grad f_i(\bar{p}),v(\bar{p})\rangle \geq 0.
\]
Combining last inequality with \eqref{eq:conv1}, we have
\[
\max_{i\in I}\langle\grad f_i(\bar{p}),v(\bar{p})\rangle + (1/2)\|v(\bar{p})\|^{2}=0.
\]  
Therefore, again from Definition \ref{def:sddf1} and Lemma \ref{dir:md2} it follows that $\bar{p}$ is a Pareto critical and the proof is concluded.
\end{proof}
\begin{remark}
If the sequence $\{p^{k}\}$ begins in a bounded level set, for example, if 
\[
L_F(F(p_0)):=\{ p\in M: F(p)\preceq F(p_0)\},
\]
is a bounded set, and being $F$ a continuous function, Hopf-Rinow's theorem assures that $L_F(F(p_0))$ is a compact set. So,  item $i$ of Theorem~\ref{resconv1}  implies that  $\{p^{k}\}\subset L_F(F(p_0))$ and consequently  $\{p^{k}\}$ is bounded. In particular,  $\{p^{k}\}$ has at least one accumulation point. Therefore, Theorem~\ref{resconv1} extends for vector optimization the results of Theorem~$5.1$ of \cite{XLO1998}. See also Remark~$4.5$ of \cite{Gabay1982}.
\end{remark}

\subsection{Full convergence}\label{S:4.2}
In this section under the quasi-convexity assumption on $F$ and not-negative curvature for $M$, full convergence of the steepest descent method is obtained.
\begin{definition}
Let $H:M\to\mathbb{R}^m$ be a vectorial function.
\begin{itemize}
\item [i)] $H$ is called convex on $M$ if for every $p,q\in M$ and every geodesic segment $\gamma:[0,1]\to M$ joining $p$ to $q$ (i.e., $\gamma(0)=p$ and $\gamma(1)=q$), it holds
\[
H(\gamma(t))\preceq (1-t)H(p)+tH(q),\qquad t\in [0,1].
\]
\item [ii)] $H$ is called quasi-convex on $M$ if for every $p,q\in M$ and every geodesic segment $\gamma:[0,1]\to M$ joining $p$ to $q$, it holds
\[
H(\gamma(t))\preceq \max\{H(p),H(q)\},\qquad t\in [0,1],
\]
where the maximum is considered  coordinate by coordinate.
\end{itemize}
\end{definition}
\begin{remark}\label{re:conv1}
The first above definition is a natural extension of the definition of convexity while the second is an extension of  a characterization of the definition of quasi-convexity, of the Euclidean space to the Riemannian context. See Definition $6.2$ and Corollary~$6.6$ of \cite{Luc1989}, pages $29$ and $31$ respectively. Note that the above definitions are equivalent, respectively, $H$ to be convex and quasi-convex along every geodesic segment. Thus, when $m=1$ these definitions  merge into the scalar convexity and quasi-convexity defined in  \cite{U94}, respectively. Moreover, it is immediate of the above definitions that if $H$ is convex then it is quasi-convex. In the case that $H$ is differentiable, convexity of $H$ implies that for every $p,q\in M$ and every geodesic segment $\gamma:[0,1]\to M$ such that $\gamma(0)=p$ and $\gamma(1)=q$,
\[
\grad H(p)\gamma'(0)\preceq H(q)-H(p).
\]

\end{remark}
\begin{proposition}\label{prop:qc10}
Let $H:M\to\mathbb{R}^m$ be a differentiable quasi-convex function. Then, for every $p,q\in M$ and every geodesic segment $\gamma:[0,1]\to M$ joining $p$ to $q$, it holds
\[
H(q)\preceq H(p)\quad \Rightarrow\quad \grad H(p)\gamma'(0)\preceq 0.
\]
\end{proposition}
\begin{proof}
Take $p,q\in M$ such that $H(q)\preceq H(p)$ and a geodesic segment $\gamma:[0,1]\to M$ such that $\gamma(0)=p$ and $\gamma(1)=q$. Since $H$ is quasi-convex, we have
\[
H(\gamma(t))\preceq H(p),\qquad t\in[0,1].
\] 
Using last inequality  the result is an immediate consequence from the differentiability of $H$.
\end{proof}
We know that criticality is necessary condition but not sufficient for optimality. However, under convexity  of the vectorial function $F$ we proved that criticality is equivalent to the weak optimality.
\begin{definition}
A point $p^*\in M$ is a weak Pareto optimal of $F$ if there is no $p\in M$ with $F(p)\prec F(p^*)$.
\end{definition}
\begin{proposition}\label{prop:optm1}
Let $H:M\to {\mathbb R}^m$  be a convex continuously differentiable function. Then, $p\in M$ is a Pareto critical of $H$, i.e., 
\[
\mbox{Im}(\grad H(p))\cap (-{\mathbb R}^{m}_{++})=\emptyset,
\]
if and only if  $p$ is a weak Pareto optimal of $H$.
\end{proposition}
\begin{proof}
Let us suppose that $p$ is Pareto critical of $H$. Assume by contradiction that $p$ is not weak Pareto optimal of $H$. Since $p$ is not weak Pareto optimal, there exists $\tilde{p}\in M$ such that
\begin{equation}\label{eq:cdiff2}
H(\tilde{p})\prec H(p).
\end{equation}
Let $\gamma:[0,1]\to M$ be a geodesic segment joining $p$ to $\tilde{p}$ (i.e., $\gamma(0)=p$ and $\gamma(1)=\tilde{p}$). As $H$ is differentiable and convex, the last part of Remark~\ref{re:conv1} and \eqref{eq:cdiff2} imply that
\[
\grad H(p)\gamma'(0)\preceq H(\tilde{p})-H(p)\prec 0.
\]
But this contradicts the fact of $p$ to be Pareto critical of $H$, and the first part is concluded.  

Now, let us suppose that $p$ is weak Pareto optimal of $H$.  Assume by contradiction that $p$ is not Pareto critical of $H$. Since $p$ is not Pareto critical, then $\mbox{Im}(\grad H(p))\cap (-{\mathbb R}^{m}_{++})\neq\emptyset$, that is,  there exists $v\in T_pM$ a descent direction for $F$ at $p$. Hence, from the differentiability of $H$, we have
\[
\lim_{t\to 0^+}\frac{H \left(exp _{p}(tv)\right)-H(p)}{t}=\grad H(p)v\prec 0,
\]
which implies that there exists $\delta>0$ such that 
\[
H\left(exp _{p}(tv)\right)\prec  H(p)+ t\grad H(p)v, \qquad t\in (0,\, \delta).
\]
Since $v$ is a descent direction for $F$ at $p$ and $t\in (0,\, \delta)$ we have $t\grad H(p)v\prec 0$. So,   the last vector inequality yields
\[
H\left(exp _{p}(tv)\right)\prec  H(p), \qquad t\in (0,\, \delta),
\] 
contradicting the fact of $p$ to be weak Pareto optimal of $H$, which concludes the proof.
\end{proof}

\begin{definition}\label{Fejconv01}
A sequence $\{q^k\}\subset M$ is  quasi-Fej\'er convergent to a  nonempty set $U$ if, for all $p\in U$, there exists a sequence $\{\epsilon_k\}\subset\mathbb{R}_+$ such that
\[
\sum_{k=0}^{+\infty}\epsilon_k<+\infty, \qquad d^2(q^{k+1},q)\leq d^2(q^k,q)+\epsilon_k, \qquad k=0,1,\ldots.
\]
\end{definition}
In the next lemma we recall the called quasi-Fej\'er convergence theorem.
\begin{lemma}\label{Fejconv1}
Let $U\subset M$ be a nonempty set and $\{q^k\}\subset M$ a sequence quasi-Fej\'er convergent. Then, $\{q^k\}$ is bounded. Moreover, if an accumulation point $\bar{q}$ of $\{q^k\}$ belongs to $U$, then the whole sequence $\{q^k\}$ converges to $\bar{q}$ as $k$ goes to $+\infty$.
\end{lemma}
\begin{proof}
Analogous to the proof of Theorem~$1$ in Burachik et al. \cite{Bur1995}, by  replacing the Euclidean distance by the Riemannian distance $d$.
\end{proof}
Consider the following set 
\begin{equation}\label{eq:qc1}
U:=\{p\in M: F(p)\preceq F(p^{k}),\;\; k=0,1, \ldots\}.
\end{equation}
In general, the above set may be an empty set. To guarantee that $U$ is nonempty, an additional assumption on the sequence $\{p^{k}\}$ is needed. In the next remark we give a such condition.
\begin{remark}
If  the sequence $\{p^{k}\}$ has an accumulation point, then  $U$ is nonempty. Indeed, let $\bar{p}$ be an accumulation point of the sequence $\{p^{k}\}$. Then, there exists a subsequence $\{p^{k_j}\}$ of $\{p^{k}\}$ which converges to $\bar{p}$. Since $F$ is continuous $\{F(p^{k})\}$ has $F(\bar{p})$ as an accumulation point. Hence, using that $\{F(p^{k})\}$ is a decreasing sequence (see item i of Theorem~\ref{resconv1})  usual arguments show easily  that the whole sequence  $\{F(p^{k})\}$ converges to $F(\bar{p})$  and there holds
\[
F(\bar{p})\preceq F(p^{k}),\qquad  k=0,1, \ldots,
\]
which implies that $\bar{p}\in U$, i.e., $U\neq\emptyset$.
\end{remark}

In the next lemma we presented the main result of this section. It is fundamental in the proof of the global convergence result of the sequence $\{p^{k}\}$.

\begin{lemma}\label{Fejconv2}
Suppose that $F$ is quasi-convex, $M$ has not-negative curvature and $U$, defined in \eqref{eq:qc1}, is nonempty. Then, for all $\tilde{p}\in U$, the inequality there holds:
\[
d^2(p^{k+1},\tilde{p})\leq d^2(p^{k},\tilde{p})+t_k^2\|v^k\|^2.
\]  
\end{lemma}
\begin{proof}
Consider the geodesic hinge $(\gamma_1,\gamma_2,\alpha)$, where $\gamma_1$ is a normalized minimal geodesic segment joining $p^{k}$ to $\tilde{p}$, $\gamma_2$ is the geodesic segment joining $p^{k}$ to $p^{k+1}$ such that $\gamma_2'(0)=t_kv^k$ and $\alpha=\angle(\gamma'_1(0),v^k)$. By the law of cosines (Theorem \ref{law-cosines}), we have
\[
d^2(p^{k+1},\tilde{p})\leq d^2(p^{k},\tilde{p})+t_k^2\|v^k\|^2-2d(p^{k},p)t_k\|v^k\|\cos\alpha,\qquad k=0,1,\ldots.
\]
Thus, taking into account that $\cos(\pi-\alpha)=-\cos\alpha$ and $\langle-v^k,\gamma_1'(0)\rangle=\|v^k\|\cos(\pi-\alpha)$, above vector inequality becomes
\[
d^2(p^{k+1},\tilde{p})\leq d^2(p^{k},\tilde{p})+2d(p^{k},\tilde{p})t_k\langle -v^k,\gamma_1'(0)\rangle,\qquad k=0,1,\ldots.
\]
On the other hand, from \eqref{eq:vk}, Definition \ref{def:sddf1} and Lemma \ref{dir:md1}, there exist $\alpha^{k}_i\geq 0$, with $i\in I_k:=I(p^{k},  v^k)$, such that 
\[
v^{k}=-\sum\limits_{i\in  I_k}\alpha_i\grad f_i(p^{k}), \qquad \sum\limits_{i\in  I_k}\alpha^{k}_i=1,\qquad k=0,1,\ldots.
\]
Hence, last vector inequality yields 
\begin{equation}\label{eq:contv10}
d^2(p^{k+1},\tilde{p})\leq d^2(p^{k},\tilde{p})+2d(p^{k},\tilde{p})t_k\sum_{i\in  I_k}\alpha^k_i\langle \grad f_i(p^{k}),\gamma_1'(0)\rangle,\quad k=0,1,\ldots.
\end{equation}
Since $F$ is quasi-convex and $\tilde{p}\in U$, from Proposition \ref{prop:qc10} with $H=F$, $p=p^{k}$, $q=\tilde{p}$ and $\gamma=\gamma_1$, we have
\[
\grad F(p^{k})\gamma_1'(0)\preceq 0, \qquad k=0,1,\ldots,
\]
or equivalently, 
\begin{equation}\label{eq:contv11}
\langle\grad f_i(p^{k}),\gamma_1'(0) \rangle\leq 0, \qquad i=1, \ldots, m,  \quad k=0,1,\ldots.
\end{equation}
Therefore, for combining \eqref{eq:contv10} with \eqref{eq:contv11}, the lemma follows. 
\end{proof}

\begin{proposition} \label{pr:cf}
If $F$ is quasi-convex, $M$ has not-negative curvature and $U$, defined in \eqref{eq:qc1}, is a nonempty set, then the sequence $\{p_k\}$ is Fej\'er convergent to $U$.
\end{proposition}
\begin{proof}
For simplify the notation define the scalar function  $\varphi:\mathbb{R}^m\to\mathbb{R}$ as follows
\[
\varphi(y)=\max_{i\in I}\langle y,e_i\rangle, \qquad I=\{ 1,  \ldots, m \}.
\]
where $\{e_i\}\subset \mathbb{R}^m$ is the canonical base of the space $\mathbb{R}^m$. It easy to see that the following properties on the function $\varphi$ hold:
\begin{equation}\label{eq:prop1}
\varphi(x+y)\leq\varphi(x)+\varphi(y),\qquad\varphi(tx)=t\varphi(x), \qquad x,y\in\mathbb{R}^m, \quad t\geq 0.
\end{equation}
\begin{equation}\label{eq:prop2}
x\preceq y\quad \Rightarrow \quad  \varphi(x)\leq\varphi (y),\qquad x,y\in\mathbb{R}^m.
\end{equation}
From the definition of $t_k$ in \eqref{eq:sl} and $p^{k+1}$ in \eqref{eq:xk}, we have
\[
F(p^{k+1})\preceq F(p^{k})+\beta t_k\grad F(p^{k})v^k,\qquad k=0,1\ldots.
\]
Hence, using  \eqref{eq:prop1}, \eqref{eq:prop2} and last inequality, we obtain 
\begin{equation}\label{eq:prop3}
\varphi(F(p^{k+1}))\leq\varphi(F(p^{k}))+\beta t_k\varphi(\grad F(p^{k})v^k),\qquad k=0,1\ldots.
\end{equation}
On the other hand, combining definition of $v^k$ in \eqref{eq:vk}, Definition \ref{def:sddf1}, Lemma \ref{dir:md2} and definition of $\varphi$, we conclude that
\[
\varphi(\grad F(p^{k})v^k) + (1/2)\|v^k\|^{2}<0, \qquad k=0,1\ldots,
\]
which together with \eqref{eq:prop3} implies that
\[
\varphi(F(p^{k+1}))<\varphi(F(p^{k}))-(\beta t_k/2)\|v^k\|^2,\qquad k=0,1\ldots,
\]
But this tells us that,
\[
t_k\|v^k\|^2<2[\varphi(F(p^k))-\varphi(F(p^{k+1}))]/\beta, \qquad k=0,1\ldots.
\]
As $t_k\in (0,1]$, follows that
\[
t_k^2\|v^k\|^2<2[\varphi(F(p^k))-\varphi(F(p^{k+1}))]/\beta, \qquad k=0,1\ldots.
\]
Thus, the latter inequality implies easily that  
\[
\sum\limits_{k=0}^{n}t_k^2\|v^k\|^2< 2\left[\varphi(F(p^0)-\varphi(F(p^{n+1}))\right]/\beta,\quad n>0.
\]
Take $\bar{p}\in U$. Then, $F(\bar{p})\preceq F(p^{n+1})$. So, from \eqref{eq:prop2} $\varphi(F(\bar{p}))\leq \varphi(F(p^{n+1}))$ and last inequality yields
\[
\sum\limits_{k=0}^{n}t_k^2\|v^k\|^2< 2\left(\varphi(F(p_0)-\varphi(F(\bar{p}))\right)/\beta.
\]
which implies that $\{t_k^2\|v^k\|^2\}$ is a summable sequence.
Therefore, the desired result follows from Lemma \ref{Fejconv2} combined with Definition \ref{Fejconv01}.
\end{proof}
\begin{theorem}
If $F$ is quasi-convex, $M$ has not-negative curvature and $U$, defined in \eqref{eq:qc1}, is a nonempty set, then the sequence $\{p_k\}$ converges to a Pareto critical of $F$.
\end{theorem}
\begin{proof}
From  Proposition~\ref{pr:cf}, $\{p^{k}\}$ is Fej\'er convergent to $U$. Thus Lemma \ref{Fejconv1} guarantees that $\{p^{k}\}$ is bounded and, from Hopf-Rinow' theorem, there exists $\{p^{k_s}\}$, subsequence of $\{p^{k}\}$, which converges to $\bar{p}\in M$ as $s$ goes to $+\infty$. Since $F$ is continuous and $\{F(p^{k})\}$ is a decreasing sequence (see item $i$ of Theorem \ref{resconv1}), we conclude that $F(p^{k})$ converges to $F(\bar{p})$ as $k$ goes to $+\infty$, which implies that 
\[
F(\bar{p})\preceq F(p^{k}),\qquad k=0,1,\ldots,
\]
i.e., $\bar{p}\in U$. Hence, from Lemma \ref{Fejconv1} we conclude that the whole sequence $\{p^{k}\}$ converges to $\bar{p}$ as $k$ goes to $+\infty$, and the conclusion of the proof it is consequence of the item $ii$ of Theorem~ \ref{resconv1}.
\end{proof}
\begin{corollary}\label{res:fconv100}
If $F$ is convex, $M$ has not-negative curvature and $U$, defined in \eqref{eq:qc1}, is a nonempty set, then the sequence $\{p_k\}$ converges to a weak Pareto optimal of $F$.
\end{corollary}
\begin{proof}
Since $F$ is convex, in particular, it is quasi-convex (see Remark \ref{re:conv1}). Thus the corollary is a consequence of the previous theorem and Proposition \ref{prop:optm1}.
\end{proof}
\section{Examples}\label{sec5}
In this section we present some examples of complete Riemannian manifolds with explicit geodesic curves and the steepest descent iteration of the sequence generated by the method~\ref{algor1}.
We recall  that $F:M\to\mathbb{R}^m$, $F(p):=(f_1(p),\ldots,f_m(p))$, is a differentiable function. If $(M,G)$ is a Riemannian manifold then the Riemannian gradient of $f_i$ is given by $\grad f_i(p)=G(p)^{-1}f_i'(p)$, $i\in I:=\{1,\ldots,n\}$. Hence, if $v(p)$ is the steepest descent direction for $F$ at $p$ (see Definition~\ref{def:sddf1}), from Lemma~\ref{dir:md1}, there exist constants $\alpha_i\geq 0$, $i\in I(p,v)$, such that 
\begin{equation}\label{eq:exsdm2} 
v=-\sum\limits_{i\in I(p,v)}\alpha_iG(p)^{-1}f'_i(p), \qquad \sum\limits_{i\in I(p,v)}\alpha_i=1,
\end{equation}
where $I(p,v):=\{i\in I: \langle G(p)^{-1}f'_i(p),v \rangle=\max_{i\in I}\langle G(p)^{-1}f'_i(p),v \rangle\}$.

\subsection{A steepest descent method for $\mathbb{R}_{++}^n$}
Let $M$ be the positive octant, $\mathbb{R}_{++}^n$, endowed with the Riemannian metric 
\[
M\ni v\mapsto G(p)=P^{-2}:=\diag\left(p_1^{-2},\ldots,p_n^{-2}\right),
\]
(metric induzed by the Hessian of the logarithmic barrier). Since $(M,G)$ is isometric to the Euclidean space endowed with the usual metric (see, Da Cruz Neto et al. \cite{XFLN2006}) it follows that $M$ has constant curvature equal to zero. On the other hand, it is easy to see that the unique geodesic $p=p(t)$ such that $p(0)=p^0=(p^0_1,\ldots,p^0_n)$ and $p'(0)=v^0=(v^0_1,\ldots,v^0_n)$ is given by $p(t)=(p_1(t),\ldots,p_n(t))$, where
\begin{equation}\label{eq:exsdm1}
p_j(t)=p^0_j{\rm e}^{\left(v^0_j/p^0_j\right)t}, \qquad j=1,\ldots,n.
\end{equation}
So, we conclude that $(M,G)$ is also complete. In this case, from \eqref{eq:exsdm1} and \eqref{eq:exsdm2}, there exist $\alpha_i^k\geq 0$ such that the steepest descent iteration of the sequence generated by the method~\ref{algor1} is given by 
\[
p^{k+1}_j=p^{k}_j{\rm e}^{\left(v^k_j/p^{k}_j\right)t_k}, \quad v^k_j=-\sum_{i\in I(p^k, v^k)}\alpha^k_i (p^{k}_j)^2\frac{\partial f_i}{\partial p_j}(p^{k}),\quad \sum_{i\in I(p^k, v^k)}\alpha^k_i=1,\quad i=1,\ldots,n. 
\]
\subsection{A steepest descent method for the hypercube}
Let $M$ be the hypercube $(0,1)\times \ldots \times (0,1)$ endowed with the Riemannian metric
\[
M\ni v\mapsto G(p)=P^{-2}(I-P)^{-2}:=\diag\left((p_1)^2(1-p_1)^2,\ldots,(p_n)^2(1-p_n)^2\right),
\]
(metric induzed by the Hessian of the barrier $b(p)=\sum_{i=1}^n(2p_i-1)\big(\ln
p_i-\ln(1-p_i)\big)$. The Riemannian manifold $(M,G)$ is complete and the  geodesic $p=p(t)$, satisfying $p(0)=p^0=(p^0_1,\ldots,p^0_n)$ and $p'(0)=v^00=(v^0_1,\ldots,v^0_n)$, is given by $p(t)=(p_1(t),\ldots,p_n(t))$,
\begin{equation}\label{mrsc11}
p_j(t)=\left(1/2\right)\left[1+ \tanh \left(\left(1/2\right)\dfrac{v_j}{p_j(1-p_j)}t+\left(1/2\right)\ln\left(\dfrac{p_j}{1-p_j}\right)\right)\right],\qquad j=1,...,n,
\end{equation}
where $\tanh(z):=(e^z-e^{-z})/(e^z+e^{-z})$. Moreover $(M,G)$ has constant curvature equal to zero, see Theorem 3.1 and 3.2 of \cite{PP2007}. In this case, from \eqref{mrsc11} and \eqref{eq:exsdm2}, there exist $\alpha_i^k\geq 0$ such that the steepest descent iteration of the sequence generated by the method~\ref{algor1} is given by   
\[
p^{k+1}_j=\left(1/2\right)\left[1+ \tanh \left(\left(1/2\right)\dfrac{v^k_j}{p^{k}_j(1-p^{k}_j)}t_k+\left(1/2\right)\ln\left(\dfrac{p^{k}_j}{1-p^{k}_j}\right)\right)\right],\qquad j=1,...,n,
\]
with,
\[
v^k_j=-\sum_{i\in I(p^k, v^k)}\alpha^k_i (p^{k}_j)^2(1-p^{k}_j)^2\frac{\partial f_i}{\partial p_j}(p^{k}),\quad \sum_{i\in I(p^k, v^k)}\alpha^k_i=1,\qquad j=1,\ldots,n. 
\]
\subsection{steepest descent method for the cone of positive semidefinite matrices}
Let ${\mathbb S}^{n}$ be the set
of the symmetric matrices  $n\times n$, ${\mathbb S}^{n}_{+}$ the cone of
the symmetric positive semi-definite matrices and ${\mathbb
S}^{n}_{++}$ the cone of the symmetric positive definite
matrices. Following Rothaus \cite{Rot1960}, let $M={\mathbb S}^n_{++}$ be endowed with the Riemannian metric induced by the Euclidean Hessian of $\Psi(X)=-\ln\det X$, i.e., $G(X):=\Psi''(X)$. In this case, the unique geodesic segment connecting any $X,Y\in M$ is given by
\[
X(t)=X^{1/2}\left(X^{-1/2}YX^{-1/2}\right)^tX^{1/2},\qquad t\in[0,1],
\]
see \cite{NT01}.	
More precisely, $M$ is a Hadamard manifold (with curvature not identically zero), see for example \cite{Lang 1998}, Theorem~1.2, page 325. In particular, the unique geodesic $X=X(t)$ such that $X(0)=X$ and $X'(0)=V$ is given by 
\begin{equation}\label{eq:exsdm100}
X(t)=X^{1/2}{\rm e}^{tX^{-1/2}VX^{-1/2}}X^{1/2}.
\end{equation}
Thus, from \eqref{eq:exsdm100} and \eqref{eq:exsdm2}, there exist $\alpha_j^k\geq 0$ such that the steepest descent iteration of the sequence generated by the method~\ref{algor1} is given by   
\[
X^{k+1}=(X^{k})^{1/2}{\rm e}^{t_k(X^k)^{-1/2}V^k(X^k)^{-1/2}}(X^k)^{1/2},
\]
with,
\[
V^k=-\sum_{i\in I(X^k,V^k)}\alpha^k_i X^{k}f'_i(X^k)X^k,\quad \sum_{i\in I(X^k,V^k)}\alpha^k_i=1.
\]
\begin{remark}
Under the assumption of convexity on the vector function $F$, if $(M,G)$ is the Riemannian manifod in the first or in the second example, then Corollary~\ref{res:fconv100} assures the full convergence of the sequence generated by Method~\ref{algor1}. This fact doesn't necessarily happens if $(M,G)$ is the Riemanniana manifold in the last example, since in this case $(M,G)$ has curvature not-positive, i.e., $K\leq0$. However, Theorem~\ref{resconv1} assures at least partial convergence. 
\end{remark}
\section{Final remarks}
We have extended the steepest descent method with Armijo's rule for multicriteria optimization to the Riemannian context. Full convergence is obtained under to the assumptions of quasi-convexity of the multicriteria function and non-negative curvature of the Riemannian manifold. A subject in open is to obtain the same result without restrictive
assumption on the curvature of the manifold. Following the same line of this paper, as future propose we have the extension, to the context Riemannian, of the proximal method (see  Bonnel et al. \cite{Benar2005-1}) and Newton method (see Fliege et al. \cite{Benar2009}), both for multiobjective optimization.


\begin{thebibliography}{99}
%
\bibitem{Absil2007} Absil, P. -A., Baker, C. G., Gallivan, K. A. \emph{Trust-region methods on Riemannian manifolds}. To appear in Foundations of Computational Mathematics. 7 (2007), no.(3), 303-330.
%
\bibitem{Teboulle2004} Attouch, H., Bolte, J., Redont, P., Teboulle, M. \emph{Singular Riemannian barrier methods and gradient-projection dynamical systems for constrained optimization}.  Optimization. 53 (2004), no. 5-6, 435-454.
%
\bibitem{Azagra2005} Azagra, D., Ferrera, J. L\'opez-Mesas, M. \emph{Nonsmooth analysis and Hamilton-Jacobi equations on Riemannian manifolds}. Jornal of Functional Analysis. 220 (2005), 304-361.
%
\bibitem{BP09} Barani, A., Pouryayevali, M.R.,\emph{Invariant monotone vector fields on Riemannian manifolds},  Nonlinear Analysis, Theory, Methods and Applications,  70(5), (2009), 1850-1861.
%
\bibitem{Benar2005-1}  Bonnel, H., Iusem, A. N, Svaiter, B. F., \emph{Proximal methods in vector optimization}. SIAM J. Optim. 15 (2005), no. 4, 953-970.
%
\bibitem{Bur1995} Burachik, R., Drummond, L. M. Gra\~na, Iusem, A. N., Svaiter, B. F. \emph{Full convergence of the steepest descent method with inexact line searches}.  Optimization  32  (1995),  no. 2, 137-146.
%
\bibitem{XLO1998} da Cruz Neto, J. X., de Lima, L. L., Oliveira, P. R. \emph{Geodesic algorithms in Riemannian geometry}.  Balkan J. Geom. Appl.  3  (1998),  no. 2, 89-100.
%
\bibitem{XFLN2006} da Cruz Neto, J. X., Ferreira, O. P., Luc\^ambio P\'erez, L. R., N\'emeth, S. Z. \emph{Convex-and Monotone-Transformable Mathematical Programming Problems and a Proximal-Like Point Method}. Journal of Global Optimization. 35 (2006), 53-69.
%
\bibitem{MP92} do Carmo, M. P. \emph{Riemannian Geometry}. Boston, Birkhauser, (1992).
%
\bibitem{FO06} Ferreira, O. P., \emph{ Proximal subgradient and a characterization of Lipschitz function on Riemannian manifolds} J. Math. Anal.  Appl. 313, no. 2, (2006), 587-597.
%
\bibitem{FS02} Ferreira, O. P., B. F., Svaiter, \emph{Kantorovich's Theorem on Newton's Method in Riemannian manifolds}. Journal of Complexity. 18 (2002), 304-329.
%
\bibitem{FO98}  Ferreira, O. P., Oliveira, P. R. \emph{Subgradient algorithm on Riemannian manifolds}. Journal of Optimization Theory and Applications. 97 (1998), no.1, 93-104.
%
\bibitem{Benar2009} Fliege, J.; Graña Drummond, L. M.; Svaiter, B. F. \emph{Newton's method for multiobjective optimization}. SIAM J. Optim. 20 (2009), no. 2, 602-626.
%
\bibitem{Benar2000} Fliege, J., Svaiter, B. F. \emph{Steepest descent methods for multicriteria optimization}.  Math. Methods Oper. Res.  51  (2000),  no. 3, 479-494.
%
\bibitem{Gabay1982} Gabay, D., \emph{Minimizing a Differentiable Function over a Differentiable Manifold}. Optim. Theory Appl. 37 (1982), pp. 177-219.
%
\bibitem{Iusem2004} Gra\~na Drummond, L. M., Iusem, A. N. \emph{A projected gradient method for vector optimization problems}.  Comput. Optim. Appl.  28  (2004),  no. 1, 5-29.
%
\bibitem{Benar2005} Gra\~na Drummond, L. M., Svaiter, B. F. \emph{A steepest descent method for vector optimization}.  J. Comput. Appl. Math.  175  (2005),  no. 2, 395-414.
%
\bibitem{HL93} Hiriart-Urruty, J.-B, Lemar\'echal, C. \emph{Convex analysis
and minimization algorithms I and II}, Springer-Verlag, (1993).
%
\bibitem{Kiwiel1996} Kiwiel, K. C., Murty, K. \emph{Convergence of the steepest descent method for minimizing quasiconvex functions}.  J. Optim. Theory Appl.  89  (1996),  no. 1, 221-226. 
%
\bibitem{Lang 1998} Lang, S. \emph{Fundamentals of Differential Geometry}, Springer - Verlag, (1998).
%
\bibitem{Zhu2007} Ledyaev, Yu. S., Zhu, Qiji J. \emph{Nonsmooth analysis on smooth manifolds}. Trans. Amer. Math. Soc. 359 (2007), no. 8, 3687-3732 (electronic).
%
\bibitem{Chong Li2009} Li, C., L\'opez, G., Mart\'in-M\'arquez, V. \emph{Monotone vector fields and the proximal point algorithm on Hadamard manifolds}. J. London Math. Soc., 79(2), 2009, pp.663-683.
%
\bibitem{lili09-1} Li, S. L., Li, C., Liou, Y. C., Yao, J. C. \emph{Existence of solutions for variational inequalities on Riemannian manifolds}, Nonliear Anal., 71(11), 2009, pp.5695-5706.
%
\bibitem{Luc1989} Luc, T. D., \emph{Theory of vector optimization}, Lecture Notes in Economics and Mathematical Systems, vol. 319, Springer, Berlin, 1989.
%
\bibitem{Luen1972} Luenberger, David G. \emph{The gradient projection method along geodesics}.  Management Sci.  18  (1972), 620-631.
%
\bibitem{Munier2007} Munier, J. \emph{Steepest descent method on a Riemannian manifold: the convex case}, Balkan
Jour. Geom. Appl. 12, 2 (2007), 98-106.
%
\bibitem{NT01} Nesterov, Y. E.  and Todd, M. J. \emph{On the Riemannian
Geometry Defined  by Self-Concordant Barriers and Interior-Point
Methods},  Found. Comput. Math. 2 (2002), no. 4, 333-361.
%
\bibitem{PP08} Papa Quiroz, E. A., Quispe, E. M., Oliveira, P. R. \emph{ Steepest descent method with a generalized Armijo search for quasiconvex functions on Riemannian manifolds} J. Math. Anal. Appl. 341, no.  1 (2008), 467-477.
%
\bibitem{PP2007} Papa Quiroz, E. A.; Oliveira, P. R. \emph{ New Self-Concordant Barrier for the Hypercube} J Optim Theory Appl 135 (2007) 475-490.
%
\bibitem{PO2009} Papa Quiroz, E. A.; Oliveira, P. R. \emph{Proximal point methods for quasiconvex and convex functions with Bregman distances on Hadamard manifolds}.  J. Convex Anal.  16  (2009),  no. 1, 49-69.
%
\bibitem{Rap97} Rapcs\'{a}k, T. \emph{Smooth nonlinear optimization in $R^n$}. Kluwer Academic Publishers, Dordrecht, (1997).
%
\bibitem{Rot1960} Rothaus, O. S., \emph{Domains of positivity},  Abh. Math. Sem. Univ. Hamburg. 24 (1960) 189-235.
%
\bibitem{S96} Sakai, T. \emph{Riemannian geometry}. Translations of mathematical
monographs, 149, Amer. Math. Soc., Providence, R.I. (1996).
%
\bibitem{S94} Smith, S. T. \emph{Optimization techniques on Riemannian Manifolds}. Fields Institute Communications,  Amer. Math. Soc., Providence, R.I. 3 (1994), 113-146.
%
\bibitem{U94} Udriste, C. \emph{Convex functions and optimization methods on
Riemannian manifolds}. Mathematics and its Applications. 297, Kluwer Academic Publishers (1994).
%
\bibitem{wangli09-2} J. H. Wang and C. Li, \emph{Convergence of the family of Euler-Halley type methods on Riemannian manifolds under the r-ondition}, Taiwanese J. Math., 13(2), (2009). 585-606.
%
\bibitem{wangli09} J. H. Wang and C. Li, \emph{Kantorovich's theorems of Newton's method for mappings and optimization problems on Lie groups}, IMA Numer, Anal,, (2009) to appear.
%
\bibitem{wanglihu09} J. H. Wang, S. C. Huang and C. Li, \emph{Extended Newton's method for mappings on Riemannian manifolds with values in a cone}, Taiwanese J. Math., 13, (2009), 633-656.
\end{thebibliography}
\end{document}